\documentclass{amsart}
\usepackage{graphics}
\usepackage{amscd}
\usepackage{mathdots}
\usepackage{amsfonts,amssymb,latexsym,color}
\usepackage{xypic,todonotes}
\usepackage{hyperref}
\newtheorem{theorem}{Theorem}
\newtheorem{proposition}[theorem]{Proposition}
\newtheorem{corollary}[theorem]{Corollary}
\newtheorem{lemma}[theorem]{Lemma}

\theoremstyle{definition}
\newtheorem{definition}[theorem]{Definition}
\newtheorem{remark}[theorem]{Remark}

\newtheorem*{notation}{Notation}
\newcommand{\nn}{\mathbb{N}}
\newcommand{\pp}{\mathbb{P}}
\newcommand{\qq}{\mathbb{Q}}
\newcommand{\rr}{\mathbb{R}}
\newcommand{\zz}{\mathbb{Z}}
\newcommand{\field}{\mathtt{k}}
\newcommand{\llgraffe}{\left\{}
\newcommand{\rrgraffe}{\right\}}

\newcommand{\ox}{\mathcal{O}_{X}} 
\newcommand{\odi}[1]{\mathcal{O}_{#1}}

\newcommand{\amplebu}[1]{\mathcal{O}_{X}({#1})}
\newcommand{\ggeq}{\stackrel{\scriptscriptstyle >}{\null_{\scriptscriptstyle{(=)}}}}
\newcommand{\lleq}{\stackrel{\scriptscriptstyle <}{\null_{\scriptscriptstyle{(=)}}}}
\newcommand{\ppreceq}{\stackrel{\scriptscriptstyle \prec}{\null_{\scriptscriptstyle (-)}}}
\newcommand{\ssucceq}{\stackrel{\scriptscriptstyle \succ}{\null_{\scriptscriptstyle (-)}}}
\newcommand{\ttI}{\mathtt{I}}
\newcommand{\ttJ}{\mathtt{J}}
\newcommand{\ttK}{\mathtt{K}}
\newcommand{\ttIbar}{\overline{\mathtt{I}}}
\newcommand{\kk}[2]{\textsf{k}_{\scriptscriptstyle{#1},{#2}}}
\newcommand{\kkvoid}{\textsf{k}}
\newcommand{\eps}{\mathbf{\varepsilon}}

\newcommand{\muttI}[1]{\mu_\ttI({#1})}
\newcommand{\mukk}[1]{\mu^\kkvoid({#1})}

\newcommand{\pttI}{P_\ttI}
\newcommand{\rttI}{R_\ttI}
\newcommand{\rttIa}[1]{R_{\ttI,{#1}}}

\newcommand{\gammattI}{\gamma_{\ttI}}
\newcommand{\gammattIi}[1]{\gamma_{\ttI,{#1}}}
\newcommand{\polhilb}{\mathsf{P}}

\newcommand{\hilpol}[1]{\mathsf{P}_{\scriptscriptstyle{#1}}}
\newcommand{\hilpolred}[1]{\mathsf{p}_{\scriptscriptstyle{#1}}}
\newcommand{\Lcal}{\mathcal{L}}

\newcommand{\Dsf}{\mathsf{D}}
\newcommand{\Lsf}{\mathsf{L}}

\newcommand{\efil}{E^{\bullet}}

\newcommand{\alphafil}{\underline{\alpha}}
\newcommand{\efiltration}{\efil,\alphafil}

\newcommand{\betafil}{\underline{\beta}}

\newcommand{\zetafil}{\underline{\zeta}}

\newcommand{\filtration}[2]{{#1}^{\bullet},\underline{{#2}}}
\newcommand{\ab}{_{a,b}}
\newcommand{\abc}{_{a,b,c}}
\newcommand{\pkn}[2]{p_{{#1}}\left({#2}\right)}

\newcommand{\rk}[1]{\text{rk}({#1})}

\newcommand{\rest}[2]{{#1}_{\mid_{#2}}}
\newcommand{\ol}[1]{\overline{#1}}
\newcommand{\ul}[1]{\underline{#1}}
\newcommand{\pivot}{\mathtt{p}}
\newcommand{\fatx}[3]{f_{{#1},{#2}}({#3})}

\newcommand{\underi}{\underline{i}}

\newcommand{\sigmatiny}{\tiny{\Sigma}}
\newcommand{\sigmaAtx}[3]{\Sigma_{{#1},{#2}}({#3})}

\newcommand{\aupleset}{A^{\scriptscriptstyle{ord}}}
\newcommand{\maxp}[2]{h_{{#1},{#2}}}
\title{Reduction of the semistability condition for tensors}
\author[A. Lo Giudice]{A. Lo Giudice$^{\S}$}
\thanks{The Author ${\S}$ is supported by the fapesp post-doctoral grant number 2013/20617-2}
\author[A. Pustetto]{A. Pustetto$^{\dag}$}
\thanks{The Author ${\dag}$ is supported by the post-doctoral grant 00006293 from Pontificia Universidad Javeriana}
\dedicatory{
$\S$ IMECC - UNICAMP, Department of Mathematics, Rua S{\'e}rgio Buarque de Holanda 651, Bar\~ao Geraldo, Campinas, SP - Brazil CEP 13083-859.\\
E-mail: {\tt alessiologiudic@gmail.com}\\
\bigskip
$\dag$ ​Pontificia Universidad Javeriana, Department of Mathematics, Carrera 7 \#40-62, Bogot\'a, Colombia.\\
E-mail: {\tt ex-apustetto@javeriana.edu.co, pustetto@gmail.com}
}
\date{\today}
\begin{document}
\begin{abstract}
In this article we study a special class of vector bundles, called tensors. A tensor consists of a vector bundle $E$ over a smooth irreducible projective variety and a morphism of vector bundles $\varphi$. As for classical vector bundles, there exists a notion of stability for these objects given in terms of filtrations of the vector bundle $E$. The aim of the present paper is to prove that if a destabilizing filtration  is "too" long then there exists a shorter subfiltration which destabilizes as well. Moreover, we describe some related combinatorial problems, which arise from the description of a tensor $(E,\varphi)$ or, more precisely, a filtration of $E$ as a $a$-dimensional matrix. Eventually, as example we study semistable tensors on the projective line.

\end{abstract}
\maketitle
\section{Introduction}
A tensor consists, roughly speaking, of a (coherent) sheaf $E$ over a smooth variety $X$, ``decorated'' with a morphism $\varphi$  from $(E^{\otimes a})^{\oplus b}$ to $(\det E)^{\otimes c}\otimes\Dsf_u$ where $\Dsf_u$ is a torsion free sheaf over $X$ (see Definition \ref{def-tensor}). A slightly different notion of sheaf decorated with a morphism was introduced by Schmitt, (see, for example \cite{Sch-singular}, \cite{Sch-closerlook}, \cite{Sch-global-bound}) while the more general notion of tensor was introduced by G\'omez and Sols in \cite{GS-1}. In both cases, such objects share the same semistability condition and gain their importance because they include many types of sheaves such as principal bundles, framed bundles, Higgs bundles, orthogonal and symplectic sheaves, and many others.

Recently, using this formalism, G\'omez, Langer, Schmitt and Sols construct the moduli spaces of semistable principal bundles over smooth projective varieties over algebraically closed fields of positive characteristic \cite{GSLS}. The semistability notion is very important for sheaves and it plays a fundamental role in the construction of their moduli space. Unfortunately the semistability condition for tensors, as well as the slope semistability condition, is quite complicated and has to be checked over all weighted filtrations, of any length, of the given sheaf $E$ (see Definition \ref{def-ss}).\\

Let $(E,\varphi)$ be a tensor. We associate to a weighted filtration $(\efiltration)_\ttI$, indexed by the set of indexes $\ttI$, a matrix $M_\ttI(\efiltration)$ which remembers the behavior (i.e. to be zero or not) of $\varphi$ over the given filtration. These matrices, as well as the behavior of $\varphi$ over the filtration, are uniquely determined by particular elements of the matrix, called ``pivots'' (see Section \ref{sec-associated-matrix} for the constructions and definitions), and more pivots determine a matrix more complicated is the behavior of $\varphi$ over the given filtration. As main result of this paper (Theorem \ref{teo-main}) we prove that if the length of a destabilizing filtration is (strictly) greater than the number of pivots associated to it, then exists a proper subfiltration which destabilize as well.
%We say that a filtration indexed by $\ttI$ splits if exists two proper subsets of indexes $\ttJ,\ttK\subsetneq\ttI$ such that their union is $\ttI$ and the sum of the semistability conditions of the two subfiltrations indexed by $\ttJ$ and $\ttK$ is equal to the semistability condition of the whole filtration (Equation \eqref{eq-splits}). The condition of given Theorem \ref{teo-main} is strict in the sense that exists non-splitting filtrations of length equal to the number of pivots associated to it. Moreover,the theorem is one-direction in the sense that there are also splitting filtrations of length less than the number of its pivots.
The number of pivots associated to a filtration is proportional to the complexity of the behavior of $\varphi$ over the given filtration. As consequence of Theorem \ref{teo-main} we obtain a reduction of the (semi)stability condition: it is enough to check the (semi)stability of tensors over subsheaves and filtrations with ``enough'' pivots, instead of over all filtrations. Moreover, in Section \ref{sec-combinatorial-considerations}, we investigate some combinatorial problems rising from the matrix $(\efiltration)_\ttI$ associated to a given filtration $(\efiltration)_\ttI$, and in particular we determine the maximum number of pivots, which could determine such a matrix, as a function of the rank and the type of the tensor $(E,\varphi)$. Eventually, in Section \ref{sec-rank-3-tensors-on-p1}, we study rank $3$ semistable tensors on $\pp^1$.\\

\begin{notation}
We use the convention that whenever “(semi)stable” and “$\lleq$” (resp. $\ggeq$) appear in a sentence, two statements should be read: one with “semistable” and “$\leq$” (resp. $\geq$) and another with “stable” and “$<$” (resp. $>$).\\
All the schemes in the paper are locally noetherian. A variety is an irreducible and reduced separated scheme of finite type over an algebraically closed field $\field$ of characteristic zero.
\end{notation}

\bigskip
{\noindent\bf Acknowledgment.} We thank Professors Ugo Bruzzo, Marcos Jardim and Beatriz Grana Otero for useful discussions.  
\bigskip

\section{Semistability conditions}\label{sec-semistability-conditions}

Let $(X,\amplebu{1})$ be an $n$ dimensional polarized smooth variety over an algebraically closed field $\field$. A family $\{\Dsf_u\}_{u\in R}$ of locally free sheaves over $X$ parametrized by a scheme $R$ is a locally free sheaf $\Dsf$ on $X\times R$, and for a given closed point $u\in R$, we denote by $\Dsf_u$ the restriction to the slice $X\times\{u\}$.\\

From now on we fix a polynomial $\polhilb$ of degree $n$ and integer numbers $a,b,c,d,r$ with $a,b,r\geq1$ and $c\geq0$.
\begin{definition}[\cite{GS-1} Definition $1.1$]\label{def-tensor}
 A \textbf{tensor} of type $(a,b,c,\Dsf,R)$ over $X$ is a triple $(E,\varphi,u)$ where $E$ is a coherent sheaf with Hilbert polynomial $\hilpol{E}=\polhilb$, degree $\deg (E)= d$ and  rank $\rk{E}=r$, $R$ is a scheme, $\Dsf$ is a locally free sheaf over $X\times R$ and $\varphi$ is  morphism
\begin{equation*}
 \varphi:( E^{\otimes a})^{\oplus b}\longrightarrow(\det E)^{\otimes c}\otimes\Dsf_u
\end{equation*}
not identically zero.

Sometimes we will simply call these objects tensors instead of tensors of type $(a,b,c,\Dsf,R)$ if the input data are clear by the context.
\end{definition}

\begin{remark}\label{rem-abc-equivale-ab}
The notion of tensor generalizes the notion of decorated sheaf introduced by Schmitt (\cite{Sch-global-bound}) and studied by the authors in \cite{AlePu2}. We recall the definition of the latter:
\begin{definition}\label{def-decorato}
A \textbf{decorated sheaf} of type $(a,b,c,\Lsf)$ over $X$ is the datum of a torsion free sheaf $E$ over $X$ and a non-zero morphism
\begin{equation}\label{eq-def-decorato}
\varphi:E\abc\doteqdot (E^{\otimes a})^{\oplus b}\otimes(\det E)^{\otimes -c}\longrightarrow\Lsf,
\end{equation}
\end{definition}
where $\Lsf$ is a line bundle over $X$.\\
The morphism $\varphi:E\abc\to\Lsf$ induces a morphism $E\ab\to(\det E)^{\otimes c}\otimes\Lsf$. By an abuse of notation, we still refer to the latter as $\varphi$. In this context it is easy to see that a decorated sheaf of type $(a,b,c,\Lsf)$ corresponds (uniquely up to isomorphisms) to a tensor of type $(a,b,c,\Dsf)$, where we have chosen $R=\{\text{pt}\}$ and $\Dsf$ the pullback of $\Lsf$ over $X\times R$.
\end{remark}

\begin{remark}
 The category $\mathfrak{C}^{\Lcal}_{a,b,c,\Dsf,R}$ of tensors with fixed determinant $\det E\simeq\Lcal$ of type $(a,b,c,\Dsf,R)$ is equivalent to the category $\mathfrak{C}^{\Lcal}_{a,b,0,\pi_X^*\Lcal^{\otimes c}\otimes\Dsf,R}$ of tensors with fixed determinant of type $(a,b,0,\pi_X^*\Lcal^{\otimes c}\otimes\Dsf,R)$, where $\pi_X^{\phantom a}:X\times R\to X$ is the projection. Indeed it is easy to see that map
 \begin{align*}
  \mathfrak{C}^{\Lcal}_{a,b,c,\Dsf,R} & \longrightarrow\mathfrak{C}^{\Lcal}_{a,b,0,\pi_X^*\Lcal^{\otimes c}\otimes\Dsf,R}\\
  (E,\varphi) & \longmapsto(E,\varphi)
 \end{align*}
 is an equivalence of categories.\\
\end{remark}

Since we are interested in studying the semistability condition of a given tensor, and not of families, from now on we will consider tensors of type $(a,b,0,\Dsf,R)$ with $R=\{\text{pt}\}$. Therefore, from now on, $\Dsf$ will be regarded as a torsion free sheaf over $X$ and we will denote by $(E,\varphi)$ the triple $(E,\varphi,\text{pt})$ and by $(a,b,\Dsf)$ the quintuple $(a,b,0,\Dsf,\{\text{pt}\})$. If $X$ is not smooth the determinant of $E$ could be not defined, so it is not possible to define tensors having $c\neq0$ over a non-smooth variety. As we restricted considering only tensors of type $(a,b,\Dsf)$, from now on we admit $X$ to be singular. The definitions of semistability and $\kkvoid$-semistability, that we are just about to introduce, are the same both for tensors $(E,\varphi,u)$ of type $(a,b,c,\Dsf,R)$ both for tensors of type $(a,b,\Dsf)$. For these reasons and for simplicity's sake we will give such definitions only for the latter.\\

Let $(E,\varphi)$ be a tensor of type $(a,b,\Dsf)$, consider the following filtration
\begin{equation}
 \efil:\qquad 0\subsetneq E_{i_1}\subsetneq\dots\subsetneq E_{i_s}\subsetneq E_{r}=E
\end{equation}
of saturated subsheaves of $E$, and let $\alphafil=(\alpha_{i_1},\dots,\alpha_{i_s})$ be a vector of positive rational numbers. Finally let us denote by $\ttI=\{i_1,\dots,i_s\}$ the set of indexes appearing in the filtration (we request that the set of indexes satisfies the property that for any $j$, $i_j<i_{j+1}$) and by $|\ttI|$ its cardinality. We will refer to the pair $(\efiltration)_\ttI$ as \textbf{weighted filtration} of $E$ indexed by $\ttI$, or simply weighted filtration. A weighted filtration defines the following polynomial
\begin{equation}\label{eq-def-P-di-filtrazione}
 \pttI(\efiltration)\doteqdot\sum_{i\in\ttI}\alpha_i\left(\hilpol{E}\cdot\rk{E_i}-\rk{E}\cdot\hilpol{E_i}\right),
\end{equation}
and the rational number
\begin{equation}\label{eq-def-slope-ss1}
 L_\ttI(\efiltration)\doteqdot\sum_{i\in\ttI}\alpha_i\left(\deg{E}\cdot\rk{E_i}-\rk{E}\cdot\deg{E_i}\right),
\end{equation}
where $\hilpol{E_i}$ denotes the Hilbert polynomial of $E_i$.
Finally we associate to $(\efiltration)_\ttI$ the following rational number also depending on $\varphi$,
\begin{equation}\label{eq-def-mu}
 \muttI{\efiltration;\varphi}\doteqdot -\min_{i_1,\dots,i_a\in\ttIbar}\{ \gammattIi{i_1}+\dots+\gammattIi{i_a} \,|\, \rest{\varphi}{(E_{i_1}\otimes\dots\otimes E_{i_a})^{\oplus b}}\not=0\},
\end{equation}
where $\ttIbar\doteqdot\ttI\cup\{r\}$ and
\begin{align}
  \gammattI & =(\gammattIi{1},\dots,\gammattIi{r})\nonumber\\ 
             & \doteqdot \sum_{i\in\ttI} \alpha_i (\underbrace{\rk{E_i}-r,\dots,\rk{E_i}-r}_{\rk{E_i}\text{-times}},\underbrace{\rk{E_i},\dots,\rk{E_i}}_{r-\rk{E_i}\text{-times}}).
\end{align}\\
The notion of semistability for a tensor depends on a stability parameter $\delta$, which essentially measures how far a semistable tensor is from being semistable in the usual way. The parameter $\delta$ is a rational polynomial $\ol{\delta} x^{n-1}+\delta_{n-2} x^{n-2}\dots+\delta_{1}x+\delta_{0}$ 
with positive leading coefficient $\ol{\delta}>0 $.

\begin{definition}[\textbf{Semistability}]\label{def-ss}
 Let $(E,\varphi)$ be a tensor of type $(a,b,\Dsf)$. Then $(E,\varphi)$ is $\mathbf{\delta}$-\textbf{(semi)stable} if for any weighted filtration $(\efiltration)_\ttI$ the following inequality holds:
\begin{equation}\label{eq-def-ss}
 \pttI^\mu(\efiltration;\varphi) = \pttI(\efiltration)+\delta\muttI{\efiltration;\varphi}\ssucceq0.
\end{equation}
 The tensor is \textbf{slope $\ol{\delta}$-(semi)stable} if
\begin{equation}\label{eq-def-slope-ss}
 L_\ttI^\mu(\efiltration;\varphi) = L_\ttI(\efiltration)+\ol{\delta}\muttI{\efiltration;\varphi}\ggeq0.
\end{equation}
Sometimes we will write $P^\mu$ (resp. $L^\mu$) instead of $\pttI^\mu$ (resp. $L_\ttI^\mu$) if the set of indexes $\ttI$ is understood. Moreover, from now on, we will write (semi)stable (resp. slope (semi)stable), instead of $\delta$-(semi)stable (resp. slope $\ol{\delta}$-(semi)stable), unless we want to stress the reader's attention on the parameter $\delta$ (resp. $\ol{\delta}$).
\end{definition}

\begin{remark}
Similarly to the case of sheaves, we have the following chain of implications (see \cite{GS-1})
\begin{equation*}
\text{slope }\ol{\delta}(n-1)!\text{-stable}\Rightarrow\delta\text{-stable}\Rightarrow\delta\text{-semistable}\Rightarrow\text{slope }\ol{\delta}(n-1)!\text{-semistable}.
\end{equation*}
\end{remark}
\begin{remark}\label{rem-diamond}
 \begin{enumerate}
  \item Let $(\efiltration)_\ttI$ be a weighted filtration indexed by $\ttI$ and let $\muttI{\efiltration;\varphi}=-(\gammattIi{i_1}+\dots+\gammattIi{i_a})$. Then there exists a permutation $\sigma:\{i_1,\dots,i_a\}\to\{i_1,\dots,i_a\}$ such that $\rest{\varphi}{(E_{\sigma(i_1)}\otimes\dots\otimes E_{\sigma(i_a)})^{\oplus b}}\not=0$. %We can say that, although the morphism $\varphi$ is not symmetric, the semistability condition has a certain symmetric behavior.
  \item From now on we will write
\begin{equation*}
 \rest{\varphi}{(E_{i_1}\diamond\dots\diamond E_{i_a})^{\oplus b}}\not=0
\end{equation*}
if there exists a permutation $\sigma:\{i_1,\dots,i_a\}\to\{i_1,\dots,i_a\}$ such that $\rest{\varphi}{(E_{\sigma(i_1)}\otimes\dots\otimes E_{\sigma(i_a)})^{\oplus b}}\not=0$.
 \end{enumerate}
\end{remark}

\begin{definition}
 Let $(E,\varphi)$ by a tensor of type $(a,b,\Dsf)$ and $(\efiltration)_\ttI$ be a weighted filtration of $E$ indexed by $\ttI$. For any $i\in \ttI$ let $(0\subset E_i\subset E,\alpha_i)$ be the induced length one filtration. We will say that $(\efiltration)_\ttI$ is \textbf{non-critical} if
 \begin{equation*}
  \muttI{\efiltration;\varphi} = \sum_{i\in\ttI} \mu_{\{i\}}(0\subset E_i\subset E,\alpha_i;\varphi),
 \end{equation*}
 and \textbf{critical} otherwise.
 
 We say that the filtration $(\efiltration)_\ttI$ \textbf{splits} (or, analogously, that it is a splitting filtration) if exist two proper subsets of indexes $\ttJ,\ttK\subsetneq\ttI$ with $\ttJ\cup\ttK=\ttI$, and two vectors of positive rational numbers $\betafil\in\qq_+^{|\ttJ|}$ and $\zetafil\in\qq_+^{|\ttK|}$ such that
 \begin{equation*}
  \pttI^\mu(\efiltration;\varphi)=P^\mu_\ttJ(\efil,\betafil;\varphi)+\pttI^\mu(\efil,\zetafil;\varphi).
 \end{equation*}
 Otherwise, we say that it is \textbf{non-splitting}.\\
\end{definition}

Now we will introduce another notion of semistability for tensors that will be useful in the future. This notion was already introduced and studied in \cite{Pu} in the case of decorated sheaves.\\

Let $(E,\varphi)$ be a tensor and let $F$ be a subsheaf of $E$, then define
\begin{equation*}
  \kk{F}{E}=\kkvoid(F,E;\varphi)=\begin{cases}
                                                          a \text{ if } \rest{\varphi}{F\ab}\neq0\\
                                                          k \text{ if } \rest{\varphi}{F^{\diamond k}\diamond E^{\diamond (a-k)}}\neq0\text{ and }\rest{\varphi}{F^{\diamond(k+1)}\diamond E^{\diamond (a-k-1)}}=0\\
                                                          0 \text{ otherwise,}
                                                         \end{cases}
\end{equation*}

\begin{definition}[\textbf{$\kkvoid$-semistability}]
Let $(E,\varphi)$ be a tensor of type $(a,b,\Dsf)$ of positive rank; we will say that $(E,\varphi)$ is $\mathbf{\kkvoid}$-(semi)stable or slope $\mathbf{\kkvoid}$-(semi)stable if and only if for any proper subsheaf $F$ the following inequalities hold
\begin{align*}
& \mathbb{\kkvoid}\textbf{-(semi)stable } \quad & \rk{E}(\hilpol{F}-\delta\kk{F}{E}) & \ppreceq\rk{F}(\hilpol{E}-a\delta),\\
& \textbf{slope }\mathbb{\kkvoid}\textbf{-(semi)stable } \quad & \rk{E}(\deg(F)-\ol{\delta}\kk{F}{E}) & \lleq\rk{F}(\deg(E)-a\ol{\delta}).
\end{align*}
If $E$ is torsion free and $F$ is a proper subsheaf let us define $\mukk{F}=\mu(F)-\frac{\ol{\delta}\kk{F}{E}}{\rk{F}}$ and $\hilpolred{F}^\kkvoid = \hilpolred{F}-\frac{\delta\kk{F}{E}}{\rk{F}}$, where $\mu(F)=\frac{\deg(F)}{\rk{F}}$ and $\hilpolred{F}=\frac{\hilpol{F}}{\rk{F}}$, then the above conditions become
\begin{align*}
\hilpolred{F}^\kkvoid & \ppreceq\hilpolred{E}^\kkvoid\\
\mukk{F} & \lleq \mukk{E},
\end{align*}
respectively.
\end{definition}

\begin{remark}
 Let $(E,\varphi)$ and $F$ be as before. A straightforward calculation shows that
 \begin{equation*}
\mu(0\subset F\subset E,1;\varphi)=\rk{E}\,\kk{F}{E}-a\,\rk{F}.
\end{equation*}
Therefore, the $\kkvoid$-semistability condition coincides with the semistability condition for filtrations of length one. This clearly implies that
\begin{equation*}
\text{ (semi)stability } \Rightarrow \; \kkvoid\text{-(semi)stability}
\end{equation*}
and 
\begin{equation*}
\text{ slope (semi)stability } \Rightarrow \text{ slope }\kkvoid\text{-(semi)stability}
\end{equation*}
Therefore, a tensor $(E,\varphi)$ is (semi)stable (resp. slope (semi)stable) if and only if it is $\kkvoid$-(semi)stabile (resp. slope $\kkvoid$-semistable) and condition \eqref{eq-def-ss} (resp. \eqref{eq-def-slope-ss}) holds for any critical weighted filtration. 
\end{remark}
\subsection{The associated matrix}\label{sec-associated-matrix}
 Let $(E,\varphi)$ be a tensor of type $(a,b,\Dsf)$, and fix a weighted filtration $(\efiltration)_{\ttI}=(0\subsetneq E_{i_1}\subsetneq\dots\subsetneq E_{i_s}\subsetneq E_{r}=E; \alphafil=(\alpha_i)_{i\in \ttI})$ indexed by $\ttI=\{i_1,\dots,i_s\}$. \\
Let $M_{\ttI}(\efil;\varphi)$ be the associated $a$-dimensional matrix, that is, the matrix defined by the following equation,
 \begin{equation*}
m_{i_1\dots i_a}\doteqdot\begin{cases}
                          1 \text{ if } \rest{\varphi}{(E_{i_1}\diamond\dots\diamond E_{i_a})^{\oplus b}}\not=0,\\
                          0 \text{ otherwise,}
                         \end{cases}
\end{equation*}
where $i_1,\dots,i_a\in\ttIbar$. Note that $M_{\ttI}(\efil,\varphi)$ is symmetric, that is, $m_{i_1\dots i_a}=m_{\sigma(i_1)\dots\sigma(i_a)}$ for any permutation $\sigma$ in the symmetric group $S_a$.\\

\begin{definition}\label{def-partial-order}
Let $\aupleset=\{(i_1,\dots,i_a)\in\ttIbar^{a} \,|\, i_1\leq\dots\leq i_a\}$ be the set of ordered $a$-tuples. We define a partial ordering $\curlyeqprec$ over $\aupleset$ in the following way. Let $\underline{i}=(i_1,\dots,i_a)$ and $\underline{j}=(j_1,\dots,j_a)$ be two elements of $\aupleset$, then
\begin{equation*}
 \underline{i}\curlyeqprec\underline{j} \quad\Longleftrightarrow\quad i_s\geq j_s \; \text{for any } s\in \{1,\dots,a\}
\end{equation*}
We will say that two elements $\underline{i},\underline{j}\in\aupleset$ are \textbf{comparable}, and we will denote it by $\underline{i}\sim\underline{j}$, if and only if $\underline{i}\curlyeqprec\underline{j}$ or $\underline{j}\curlyeqprec\underline{i}$. We will say that they are \textbf{incomparable}, and we will denote it by $\underline{i}\not\sim\underline{j}$, otherwise.
\end{definition}

Note that, if $m_{i_1\dots i_a}=1$ for a certain $a$-tupla $(i_1,\dots,i_a)$, then it is easy to see that $m_{j_1\dots j_a}=1$ for any $(j_1,\dots,j_a)\curlyeqprec(i_1,\dots,i_a)$. Conversely, if $m_{i_1\dots i_a}=0$, then $m_{j_1\dots j_a}=0$ for any $(j_1,\dots,j_a)\curlyeqsucc(i_1\dots i_a)$. Let $\aupleset_1=\{(i_1,\dots,i_a)\in\aupleset \,|\, 
m_{i_1\dots i_a} = 1\}$, then the set $\mathtt{P}$ of $\curlyeqprec$-maximal elements of $\aupleset_1$ uniquely determine $M_{\ttI}(\efil;\varphi)$. We will call these elements \textbf{pivots} and we will denote them with the character $\pivot$. From now on we will identify $M_{\ttI}(\efil;\varphi)$ with the set of its pivots $\mathtt{P}=\{\pivot_1,\dots,\pivot_p\}$, where $\pivot_i=(\pivot_{i1},\dots,\pivot_{ia})\in\ttIbar^{a}$.\\

Let $(\efiltration)_\ttI$ be a weighted filtration and let $M_{\ttI}(\efil;\varphi)=\{\pivot_1,\dots,\pivot_p\}$ be the associated matrix. Denoting with $r_i$ the rank of $E_i$, define
\begin{align}\label{eq-def-costanti-ci-grandi-e-piccole}
& C_i\doteqdot r_i\hilpol{E}-r\hilpol{E_i}-a\delta r_i,\\
& c_i\doteqdot r_i\deg E-r\deg E_i-a\ol{\delta} r_i, \nonumber
\end{align}
 and
\begin{equation*}
\rttIa{\alphafil}=\rttI(\efiltration;\varphi)=\max_{\pivot_i\in M_\ttI(\efil;\varphi)}\left\{ \rttIa{\alphafil}(\pivot_i))\right\}
\end{equation*}
where, for an element $\underline{i}=(i_1,\dots,i_a)\in\aupleset$,
\begin{equation*}
\rttIa{\alphafil}(\underline{i})=\sum_{j=1}^a\left(\sum_{s\geq \underline{i}_{j}, s\in\ttI}\alpha_s\right).
\end{equation*}
Note that
\begin{equation*}
 \max_{\pivot_i\in M_\ttI(\efil;\varphi)}\left\{ \rttIa{\alphafil}(\pivot_i))\right\} = \max_{\underline{i}\in\aupleset}\left\{ \rttIa{\alphafil}(\underline{i})) \;|\; \rest{\varphi}{(E_{i_1}\otimes\dots\otimes E_{i_a})^{\oplus b}}\not=0 \right\},
\end{equation*}
indeed, if $\underline{i}\curlyeqprec\underline{j}$ then $\rttIa{\alphafil}(\underline{i})\leq\rttIa{\alphafil}(\underline{j})$ and the pivots are exactly the $\curlyeqprec$-maximal elements of $\aupleset_1$.\\

Using this formalism the (semi)stability condition (\ref{eq-def-ss}) is equivalent to the following,
\begin{equation*}
 \sum_{i\in\ttI} \alpha_i C_i + r\delta\rttIa{\alphafil}\ssucceq 0,
\end{equation*}
while the slope (semi)stability condition (\ref{eq-def-slope-ss}) is equivalent to the following,
\begin{equation*}
 \sum_{i\in\ttI} \alpha_i c_i + r\ol{\delta}\rttIa{\alphafil}\ggeq0.
\end{equation*}
Indeed, suppose that the minimum of $\muttI{\efiltration;\varphi}$ is attained in $(i_1,\dots,i_a)$. Then $(i_1,\dots,i_a)$ must coincides with a pivot $\pivot_j=(\pivot_{j1},\dots,\pivot_{ja})$ of $M_{\ttI}(\efil;\varphi)$ and
\begin{align*}
 \muttI{\efiltration;\varphi}= & -(\gammattIi{i_1}+\dots+\gammattIi{i_a})\\
                               = & -\left( \sum_{l\in\ttI}\alpha_l r_l - \sum_{l\geq i_1}\alpha_l r + \dots + \sum_{l\in\ttI}\alpha_l r_l - \sum_{l\geq i_a}\alpha_l r\right)\\
                               = & -a\sum_{l\in\ttI}\alpha_l r_l+r\left(\sum_{l\geq i_1}\alpha_l+\dots+\sum_{l\geq i_a}\alpha_l\right)\\
                               = & -a\sum_{l\in\ttI}\alpha_l r_l+r\left(\sum_{j=1}^a\left(\sum_{l\geq \pivot_{ij}, l\in\ttI}\alpha_l\right)\right)\\
                               = & -a\sum_{l\in\ttI}\alpha_l r_l+r\rttIa{\alphafil}(\pivot_j).
\end{align*}

So
\begin{equation*}
\pttI(\efiltration)+\delta\muttI{\efiltration;\varphi}=\sum_{i\in\ttI}\alpha_i C_i+r\delta\rttIa{\alphafil}
\end{equation*}
and
\begin{equation*}
L_\ttI(\efiltration)+\ol{\delta}\muttI{\efiltration;\varphi}=\sum_{i\in\ttI}\alpha_i c_i+r\ol{\delta}\rttIa{\alphafil}.
\end{equation*}
 
\section{Main results}

In this section we proof that, if the length of a filtration is greater than its complexity, that is the number of pivots of the corresponding matrix, then the filtration does not play a role in semistability condition. More precisely, if such a filtration destabilizes, then there exists a proper subfiltration which destabilizes as well. This theorem extends \cite[Theorem 12]{AlePu2} where the same result is given in the particular case $a=2$.
Finally, we give some special conditions whereby a fixed filtration splits.\\

%Let $(E,\varphi)$ be a tensor of type $(a,b,\Dsf)$. Let $(\efiltration)_{\ttI}$ be a filtration indexed by $\ttI=\{i_1,\dots,i_s\}$. Let $\betafil$ and $\zetafil$ be vectors, of length $s$, of positive rational numbers. We call $(\efil,\betafil)_{\ttI}$ and $(\efil,\zetafil)_{\ttI}$ \textbf{extended subfiltrations}. An extended subfiltration $(\efil,\betafil)_{\ttI}$ is a \textbf{proper subfiltration} if there exists an $i\in\ttI$ such that 
%$\beta_i=0$. In this case we consider the filtration obtained by the previous one after having deleted the subsheaf $E_i$.\\

\begin{theorem}\label{teo-main}
Let $(E,\varphi)$ be a tensor of type $(a,b,\Dsf)$ and let $(\efiltration)_\ttI$ be a weighted filtration indexed by $\ttI=\{i_1,\dots,i_s\}$. Let $$M_\ttI(\efiltration)=\{\pivot_1,\dots,\pivot_p\}=(m_{i_1\dots i_a})_{i_j\leq i_{j+1},i_j\in\{1,\dots,a\}}$$ be the associated matrix, where $\pivot_i=(\pivot_{i1},\dots,\pivot_{ia})\in\ttIbar^{a}$ is the $i$-th pivot.\\
If $s\geq p+1$ and $\pttI^\mu(\efiltration;\varphi)\ppreceq0$ (resp. $L_\ttI^\mu(\efiltration;\varphi)\lleq0$), then there exist a proper subset $\ttJ\subset\ttI$ and a vector $\betafil=(\beta_j)_{j\in\ttJ}$ of positive rational numbers such that $P_\ttJ^\mu(\filtration{E}{\beta};\varphi)\ppreceq0$ (resp. $L_\ttJ^\mu(\filtration{E}{\beta};\varphi)\lleq0$).
\end{theorem}
\begin{proof}
Without loss of generality we can suppose that $\ttI=\{1,\dots,s\}$. We prove the statement only for the (semi)stability, being the slope (semi)stable case the same. Consider the function
\begin{align*}
f:\rr^s & \longrightarrow\rr\\
\betafil & \longmapsto \betafil_\ttI C_\ttI+r\delta\rttIa{\betafil},
\end{align*}
where $\betafil_\ttI C_\ttI=\sum_{i\in\ttI}\beta_i C_i$. Clearly $f(\betafil)$ extends $\pttI^\mu(\efil,\betafil;\varphi)$ to all $\rr^s$ and coincide over the domain $\qq^s_+$ of $\pttI^\mu$. The function $\betafil_\ttI C_\ttI$ is an homogeneous linear polynomial in $\zz[\beta_1,\dots,\beta_s]$ while $\betafil\mapsto\rttIa{\betafil}$, being the maximum between homogeneous linear polynomials in the variables $\beta_1,\dots,\beta_s$, is piecewise linear. Therefore, $f(\betafil)$ is a continuous and piecewise linear function. It is linear in the regions $V_k=\{\betafil\in\rr^s \;|\; \rttIa{\betafil}(\pivot_k)\geq\rttIa{\betafil}(\pivot_l) \text{ for any }l=1,\dots,p\}$ where the maximum does not change. Note that $f(\ul{0})=0$ and $V_{kj}=V_k\cap V_j$ is a vector subspace of $\rr^s$ of dimension $s-1$, therefore $U_f^0=f^{-1}(0)$ is a piecewise linear hypersurface in $\rr^s$. We are looking for a \textit{positive} intersection of the set $U_f=\{\betafil\in\rr^s_+ \;|\; f(\betafil)\leq0\}$ with (at least) one of the coordinate hyperplanes $H_i^+=\{\beta_i=0\}\cap\rr^s_+$, where  $\rr^s_+=\{\betafil\in\rr^s \;|\; \beta_i\geq0 \;\forall i\in\ttI\}$. Indeed, if exists $\betafil'\in U_f\cap H_i^+$, then exists also a rational $\betafil$ in such intersection and the filtration $(\efil,\betafil)_\ttJ$ indexed by $\ttJ=\ttI\smallsetminus\{i\}$ is a proper subfiltration which destabilizes. So, if $U_f$ intersects a coordinate hyperplane, we are done. Suppose that this is not the case. If $U_f$ does not intersect any positive coordinate hyperplane then the restriction of $U_f^0$ to $\rr^s_+$ should be contained in a cone $W$ with vertex in zero which does not intersect any $H_i^+$, moreover the set $\{\betafil \;|\; f(\betafil)\leq0\}$ should be contained inside the space region bounded by $U_f^0$. The minimum number of linear hyperplanes necessary to define a piecewise linear hypersurface contained in a cone is $s-1$. Since $f(\betafil)$ is linear until it reaches any $V_{kj}$, $U_f^0$ could be contained in $W$ only if $p-1\geq s-1$. Therefore, if $s\geq p+1$, is not possible that $U_f^0$ is contained in $W$ and so $U_f$ intersect a positive coordinate hyperplane and we are done.
\end{proof}
\begin{remark}
Now we want to give an example to show that the previous bound for $s$ is sharp, that is, if $s=p$ then the statement of the theorem does not hold.

Let $0\subset E_1\subset E_2 \subset E_3\subset E$ be a filtration of $E$ and, as usual, we denote by $r$ and $d$ the rank and the degree of $E$ and by $r_i$ and $d_i$ the rank and degree of $E_i$; moreover, for convenience sake we assume that $r$ is a multiple of $3$. Eventually, we will fix $\ol{\delta}=1$, $a=4$ and we suppose that $c_1=-\frac{2}{3}r$, $c_2=-2r$, $c_3=-\frac{10}{3}r$. Moreover, assume that the pivots are $(1,1,r,r)$, $(2,2,2,r)$ and $(3,3,3,3)$.
In particular we get that $\kk{E_1}{E}=2$, $\kk{E_2}{E}=3$ and $\kk{E_3}{E}=4$, so the $k$-semistability conditions, $c_i+\kk{E_i}{E}\geq 0$, are satisfied for $i=1,2,3$. If we choose weights $\alpha_1=4$, $\alpha_2=2$ and $\alpha_3=6$ we get $\rttIa{\alphafil}=24$ and the semistability condition becomes
\begin{equation*}
(\alpha_1 c_1+\alpha_2c_2 +\alpha_3 c_3+24)r=\left(-\frac{8}{3}-4-20+24\right)r<0,
\end{equation*}
so that the filtration destabilizes.

Now we will calculate the semistability conditions for the length $2$ subfiltrations and we will show that they do not destabilize.

\begin{itemize}
\item Filtration indexed by $\{1,2\}$\\
\begin{itemize}
\item case $2\beta_1>\beta_2$
\begin{equation*}
 \left(-\frac{2}{3} \beta_1-2\beta_2+(2\beta_1+2\beta_2)\right)r = \left(\frac{4}{3}\beta_1\right)r \geq 0
\end{equation*}
\item case $2\beta_1<\beta_2$
\begin{equation*}
  \left(-\frac{2}{3} \beta_1-2\beta_2+(3\beta_2)\right)r = \left(\beta_2-\frac{2}{3}\beta_1\right)r \geq 0
\end{equation*}
\end{itemize}

\item Filtration indexed by $\{2,3\}$\\
\begin{itemize}
\item case $\beta_3>3\beta_2$
\begin{equation*}
 \left(-2\beta_2-\frac{10}{3}\beta_3+(4\beta_3)\right)r = \left(\frac{2}{3}\beta_3 -\beta_2\right)r \geq 0
\end{equation*}
\item case $\beta_3<3\beta_2$
\begin{equation*}
  \left(-2\beta_2-\frac{10}{3}\beta_3+(3\beta_2+3\beta_3)\right)r = \left(\beta_2-\frac{1}{3}\beta_3\right)r \geq 0
\end{equation*}
\end{itemize}

\item Filtration indexed by $\{1,3\}$\\
\begin{itemize}
\item case $\beta_1>\beta_3$
\begin{equation*}
 \left(-\frac{2}{3} \beta_1-\frac{10}{3}\beta_3+(2\beta_1+2\beta_3)\right)r = \left(\frac{4}{3}\beta_1-\frac{4}{3}\beta_3\right)r \geq 0
\end{equation*}
\item case $\beta_1<\beta_3$
\begin{equation*}
  \left(-\frac{2}{3} \beta_1-\frac{10}{3}\beta_3+(4\beta_3)\right)r = \left(\frac{2}{3}\beta_1-\frac{2}{3}\beta_3\right)r \geq 0
\end{equation*}
\end{itemize}
\end{itemize}
So there are no destabilizing length $2$ filtrations.\\

For a concrete example it is enough to consider the following rank $6$ vector bundle on a smooth curve $X$ 
\begin{equation*}
E=\amplebu{1}^{\oplus6}
\end{equation*}
and the filtration
\begin{equation*}
0\subset E_1\subset E_2\subset E_3\subset E,
\end{equation*}
where $E_1=\amplebu{1}$, $E_2=E_1\oplus\amplebu{1}^{\oplus2}$ and $E_3=E_2\oplus\amplebu{1}^{\oplus2}$.
Indeed, recalling that $c_i=dr_i-rd_i-a\ol{\delta}r_i$ (equation \eqref{eq-def-costanti-ci-grandi-e-piccole}), we get that the constant $c_1$ relative to $E_1$ is equal to $-4$, $c_2=-12$ and $c_3=-20$. Finally, define $L=\ox(m)$, with $m$ big enough such that $L(-4)$ has sections, and  denote by $V$ the last summand of $E$, that is, $E= E_3\oplus V$, then there exist nonzero morphisms 
\begin{itemize}
\item $\phi_1:E_1\otimes E_1 \otimes V\otimes V\to L$
\item $\phi_2:E_2/E_1\otimes E_2/E_1\otimes E_2/E_1 \otimes V \to L$
\item $\phi_3:E_3/E_2\otimes E_3/E_2 \otimes E_3/E_2 \otimes E_3/E_2\to L$,
\end{itemize}
define $\phi=\phi_1\oplus\phi_2\oplus \phi_3$, so that we have
$\kk{E_1}{E}=2$, $\kk{E_2}{E}=3$ and $\kk{E_3}{E}=4$. Hence, the tensor $(E,\phi_1\oplus\phi_2\oplus\phi_3,L)$ is not semistable and it cannot be destabilized by filtrations of length $2$ or $1$.
\end{remark}
\begin{proposition}
Let $(E,\varphi)$ be a tensor of type $(a,b,\Dsf)$ and $(\efiltration)_\ttI$ be a weighted filtration. If there exists an index $j\in\ttI$ such that $c_j=d r_j-r d_j-a\ol{\delta} r_j\geq0$ (resp. $C_j=\polhilb r_j-r \hilpol{E_j}-a\delta r_j\succeq0$) then there exist a proper subset $\ttJ\subset\ttI$ and a weight vector $\betafil\in\ttJ^{a}$ such that
\begin{equation*}
L^\mu_{\ttI}(\efiltration;\varphi)\geq L^\mu_{\ttJ}(\efil,\betafil;\varphi)
\end{equation*}
(resp. $P^\mu_{\ttI}(\efiltration;\varphi)\geq P^\mu_{\ttJ}(\efil,\betafil;\varphi)$), where $(\efil,\betafil)_\ttJ$ is the subfiltration indexed by $\ttJ$.
\end{proposition}
\begin{proof}
We prove the Proposition only for the slope semistability, because the proof in the the semistability case is the same.

It is easy to see that for any pair of disjoint subsets $\ttJ,\ttK$ of $\ttI$, $\rttI\leq R_\ttJ+R_\ttK$. Therefore, $L^\mu_{\ttI}(\efiltration;\varphi)=\sum_{i\in\ttI}\alpha_i c_i+r\ol{\delta}\rttIa{\alphafil}\geq \sum_{i\in\ttI,i\neq j}\alpha_i c_i+r\ol{\delta} R_{\ttJ,\alphafil'}$, where $\ttJ=\ttI\smallsetminus\{j\}$ and $\alphafil'=(\alpha_i)_{i\in\ttJ}$.
\end{proof}
Let $(E,\varphi)$ be a tensor of type $(a,b,\Dsf)$, $(\efiltration)_\ttI$ be a weighted filtration and let $M_\ttI(\efiltration;\varphi)=\{\pivot_1,\dots,\pivot_p\}$ be the associated matrix. Define
\begin{equation*}
V_{k}=\{\betafil\in\rr^s \;|\; \rttIa{\betafil}(\pivot_k)\geq\rttIa{\betafil}(\pivot_h), \; h=1,\dots,p\}.
\end{equation*}
If $\ttJ$ is a subset of $\ttI$, we denote by $V_\ttJ=\bigcap_{i\in\ttJ}V_j$. For any $k=1,\dots,p$, $\rttIa{\betafil}(\pivot_k)=\sum_{i\in\ttI}x_{ki}\beta_i$ is a linear polynomial in $\zz[\beta_1,\dots,\beta_s]$ such that $0\leq x_{kj}\leq x_{kj+1}$ for any $j\in\ttI$. This implies that $V_k$ is a linear (closed) subspace of dimension $s$ of $\rr^s$.

If the maximum of the filtration $(\efiltration)_{\ttI}$ is achieved at the pivot $\pivot_j$ we have a linear system of inequalities $\{S_i\}_{i=1,\dots,p-1}$ of the form  $\rttIa{\alphafil}(\pivot_j)\geq \rttIa{\alphafil}(\pivot_k)$ for $k\neq j$. We will denote the linear system of equations associated to $\{S_i\}_{i=1,\dots,p-1}$ by $\{\overline{S}_i\}_{i=1,\dots,p-1}$. We will write 
$S_i(\alphafil)\geq 0$ or $\overline{S}_i(\alphafil)=0$, if we want to stress which weight vector we are considering.
\begin{proposition}
Let $(E,\varphi)$ be a tensor of type $(a,b,\Dsf)$ and $(\efiltration)_\ttI$ be a weighted filtration. Using the notation introduced above, if $V_\ttI\cap\rr^s_+\neq\{0\}$ then the filtration splits.
\end{proposition}
\begin{proof}
Suppose that the maximum is attained in $\pivot_k$, i.e. for any $j=1,\dots,p$ the following inequality holds $\rttIa{\alphafil}(\pivot_k)\geq\rttIa{\alphafil}(\pivot_j)$. Let us consider the linear system  of inequalities $\{S_l\}_{l=1,\dots,p,\;l\neq k}=\llgraffe\rttIa{\alphafil}(\pivot_k)-\rttIa{\alphafil}(\pivot_j)\geq0\rrgraffe_{l=1,\dots,p,\;l\neq k}$ associated to the filtration  $(\efiltration)_{\ttI}$ and the associated linear system of equations $\{\overline{S}_l\}_{l=1,\dots,p,\;l\neq k}$. By hypothesis, the vector subspace
\begin{equation*}
V_\ttI=\{\betafil\in\rr^s \;|\; \overline{S}_l(\betafil)=0 \text{ for any }l=1,\dots,p, \; l\neq k\} 
\end{equation*}
of $\rr^s$ intersect $\rr^s_+$, so we can choose $\zetafil\in V_\ttI\cap\rr^s_+$ such that $\betafil:=\alphafil-\zetafil\in\rr^s_+$. If the vector $\alphafil$ satisfies a linear inequality $S_l(\alphafil)\geq 0$ then we have 
\begin{equation*}
S_l(\betafil)= S_l(\alphafil- \zetafil)=S_l(\alphafil)-S_l(\zetafil)=S_l(\alphafil)\geq 0, 
\end{equation*}
since $S_l$ is linear and $\zetafil$ is a solution for $\overline{S}_l$.
Because of this, the linear system associated to the filtration $(\efil,\betafil)_{\ttI}$ is the same of the one associated to $(\efiltration)_{\ttI}$, id est, the maximum for the filtration  $(\efil,\betafil)_{\ttI}$ is achieved in the same pivot as for the filtration  $(\efiltration)_{\ttI}$. Moreover, the relation $\beta_i:= \alpha_i- \zeta_i$ implies 
\begin{equation*}
\sum_{i\in\ttI} \alpha_i c_i=\sum_{i\in\ttI} \beta_i c_i +\sum_{i\in\ttI} \zeta_i c_i.
\end{equation*}
So we need to show that
\begin{equation}\label{eqR}
\rttIa{\alphafil}=\rttIa{\betafil}+\rttIa{\zetafil}.
\end{equation}
after replacing, for all $i\in\ttI$, $\beta_i$ with $\alpha_i- \zeta_i$, the terms $\alpha_i$ will delete each others, hence Equation (\ref{eqR}) becomes an equation in the variable $\zeta_i$ which is equivalent to ${\overline{S}_{l_0}}$ for a certain $l_0\in\{1,\dots,p\}$.

Finally, in order to get a decomposition in two proper filtrations we need to add the conditions $\zeta_i=0$, and $\zeta_j=\alpha_j$ (that is, $\beta_j=0$) for some $i\neq j$. At the end, we have a system with $p+1$ equations in $s$ variables which admits a non-trivial solution if $s\geq p+1$.
\end{proof}

\begin{proposition}\label{cor-un-pivot-slope-ss}
 Let $(E,\varphi)$ be a tensor of type $(a,b,\Dsf)$, $(\efiltration)_\ttI$ be a weighted filtration with $\ttI=\{i_1,\dots,i_s\}$ and suppose that the associated matrix $M_\ttI(\efil;\varphi)=\{\pivot\}$ has only one pivot. Then
\begin{equation*}
L_\ttI^\mu(\efiltration;\varphi) = \sum_{i\in\ttI} L_{\{i\}}^\mu(0\subset E_i\subset E, \alpha_i;\varphi)
\end{equation*}
and
\begin{equation*}
P_\ttI^\mu(\efiltration;\varphi) = \sum_{i\in\ttI} P_{\{i\}}^\mu(0\subset E_i\subset E, \alpha_i;\varphi)
\end{equation*}
\end{proposition}
\begin{proof}
We will prove the theorem for the slope semistability, being the semistability case the same. Let us assume that $\pivot=(p_1,\dots,p_a)\in\ttIbar^{a}$, then
\begin{equation*}
 \rttIa{\alphafil}=\rttIa{\alphafil}(\pivot)=\alpha_{p_1}+\dots+\alpha_{p_2-1}+2\alpha_{p_2}+\dots+(a-1)\alpha_{p_a-1}+a\alpha_{p_a}+\dots+a\alpha_{i_s}.
\end{equation*}
Let $k$ be an index in $\ttI$ such that $p_j\leq k < p_{j+1}$, then
\begin{equation*}
 R_{\{k\},\alpha_k}=j,
\end{equation*}
indeed, if $R_{\{k\},\alpha_k}=l>j$ then $m_{\underbrace{k,\dots,k}_{l\text{-times}},\underbrace{r,\dots,r}_{(r-l)\text{-times}}}\neq0$ and so $(\underbrace{k,\dots,k}_{l\text{-times}},\underbrace{r,\dots,r}_{(r-l)\text{-times}})$ would be another pivot. Therefore,
\begin{equation*}
\rttIa{\alphafil}=\sum_{k\in\ttI}R_{\{k\},\alpha_k}
\end{equation*}
and we are done.
\end{proof}

\section{Combinatorial considerations}\label{sec-combinatorial-considerations}
In this section we will treat some combinatorial problems related to the study of the semistability condition of tensors. Fix $\ttI=\{1,\dots,s\}$ (since we are interested in combinatorial problems we can assume, without loss of generality, that $\ttI$ consists of the first $s$ natural numbers), let us denote by $\ttIbar=\ttI\cup\{s+1\}$ and $t=|\ttIbar|$. We are interested in the cardinality of the set $\aupleset=\{(i_1,\dots,i_a)\in\ttIbar^{a} \,|\, i_1\leq\dots\leq i_a\}$.\\

Let $n,k$ be integers, denote with $\pkn{k}{n}$ the number of partitions of $n$ into exactly $k$ parts. Then
\begin{equation*}
\begin{cases}
 \pkn{k}{n} & = \pkn{k}{n-k}+\pkn{k-1}{n-1}\\
 \pkn{0}{0} & = 1\\
 \pkn{k}{n} & = 0 \text{ if }n\leq0\text{ or }k\leq0.
\end{cases}
\end{equation*}
Is a well-known fact that $\pkn{1}{n}=1$ and that $\pkn{2}{n}=\lfloor\frac{n}{2}\rfloor$.

\begin{lemma}\label{lemma-somma-uguale-non-in-relazione}
 Let $\ttI$ and $\aupleset$ be as above and let $\underline{i}=(i_1,\dots,i_a)$ and $\underline{j}=(j_1,\dots,j_a)$ be two distinct elements of $\aupleset$. If $\sum_{l=1}^a i_l = \sum_{l=1}^a j_l$ then $\underline{i}$ and $\underline{j}$ are incomparable (see Definition \ref{def-partial-order}).
\end{lemma}
\begin{proof}
 By hypothesis $\sum_{l=1}^a i_l = \sum_{l=1}^a j_l$ and $(i_1,\dots,i_a)\neq(j_1,\dots,j_a)$. Therefore, there exist two indexes $l$ and $l'$ such that $i_l\leq j_l$ and $i_{l'}\geq j_{l'}$.
\end{proof}

\begin{theorem}\label{teo-max-num-massimali}
 Let $\ttI$ and $\aupleset$ be as above and fix a subset $B$ of $\aupleset$ and let $B^\sharp$ be the subset of $\curlyeqprec$-maximal elements of $B$. Denote by $\maxp{a,t}=\max_{B\subseteq\aupleset}\{|B^\sharp|\}$, where $|\cdot|$ denotes the cardinality of a finite set. Then $\maxp{1,t}=1$, $\maxp{2,t}=\pkn{2}{t+1}$, otherwise
 \begin{equation}
  \maxp{a,t} = \max_{a\leq x\leq at}f_{a,t}(x),
 \end{equation}
 where $f_{a,t}:[a,at]\to\nn$ is defined as follows
 \begin{align}\label{eq-function}
  f_{a,t}(x) = & \;\pkn{a}{x}-\sum_{k=t+1}^{x-a+1} f_{a-1,k}(x-k),\\
  f_{a,1}(x) = & 
   \begin{cases}
    1 \text{ if } x=a\geq1,\\
    0 \text{ otherwise,}
   \end{cases}\nonumber\\ \nonumber
  f_{1,t}(x) = &
   \begin{cases}
    1 \text{ if } 1\leq x\leq t,\\
    0 \text{ otherwise.}
   \end{cases}
 \end{align}
\end{theorem}
\begin{proof}
 If $a=1$ there is nothing to prove. If $a=2$ the maximum is clearly attained by taking $B$ as the set of anti-diagonal elements, that is $B=\{(1,r),(2,r-2),\dots,(\frac{t}{2},\frac{t}{2}+1)\}$ if $t$ is even and $B=\{(1,r),(2,r-2),\dots,(\frac{t+1}{2},\frac{t+1}{2})\}$ if $t$ is odd. In any case any element of $B$ is $\curlyeqprec$-maximal, that is $B=B^\sharp$, and the cardinality of $B$ is $\lfloor\frac{t+1}{2}\rfloor$ which is exactly $\pkn{2}{t+1}$.\\
 
 Let $a \geq 3$. We want to construct a subset, $B$ of $\aupleset$ such that $B=B^\sharp$ and which realizes the maximum. The idea for constructing $B$ is to start from an $a$-tuple $\underline{i_1}\in\aupleset$, then adding $\underline{i_2}\in\aupleset$ such that $\underline{i_1}\not\sim\underline{i_2}$, then $\underline{i_3}\in\aupleset$ such that $\underline{i_3}\not\sim\underline{i_1}$ and $\underline{i_3}\not\sim\underline{i_2}$ and so on, until is not possible to add any other $a$-tuple which is incomparable with all the previous ones. In this way we get a set $B$ with the property that $B=B^\sharp$, but clearly the cardinality of $B$ depends on the choices made.\\
 
 \textbf{Claim.} \textit{Starting from an $a$-tuple $\underline{i}=(i_1,\dots,i_a)$, in order to construct a set $B$ which is the largest possible, the best choice for adding an $a$-tuple (which is incomparable with the previous ones) is by adding an $a$-tuple $\underline{j}=(j_1,\dots,j_a)$ such that $\sum_{l=1}^a i_l = \sum_{l=1}^a j_l$.\\}
 
 Indeed, thanks to Lemma \ref{lemma-somma-uguale-non-in-relazione} the added $a$-tuples are incomparable with each other. Therefore, we have to prove that the set $B$ constructed in this way is bigger than any other set constructed starting from the same $a$-tuple $\underline{i}=(i_1,\dots,i_a)$. Without loss of generality we can suppose that $i_1 = 1$.  We will prove the claim by induction on $a$. If $a=1$ there is nothing to prove. Suppose we proved the claim for $a-1$. Starting from $\underline{i}$, because of the inductive hypothesis, the best way in order to fill the first layer is adding all $a$-tuples having the first coordinate equal to $1$ and the same sum of $\underline{i}$ . An $a$-tuple $\underline{j}=(j_1,\dots,j_a)$ such that $\sum_{l=1}^a j_l > \sum_{l=1}^a i_l$ is $\curlyeqprec$ to at least one of the $a$-tuples already added. Thefore, the only possibility is adding an $a$-tuple $\underline{j}$ with $\sum_{l=1}^a j_l \leq \sum_{l=1}^a i_l$. If the inequality is strict, it is easy to see that we are missing at least an $a$-tuple between $\underline{j}$ and one of the previous ones.\\
 
 Fix an $a$-tuple $\underline{i}$, then we construct a set $B_{\underline{i}}\subset\aupleset$ as the set containing all $\underline{j}\in\aupleset$ such that $\sum_{l=1}^{a} i_l = \sum_{l=1}^{a} j_l$. By Lemma \ref{lemma-somma-uguale-non-in-relazione} we have that $B_{\underline{i}}=B_{\underline{i}}^\sharp$; moreover, by the previous claim $\max_{B\subset\aupleset}\{|B^\sharp|\} = \max_{\underline{i}\in\aupleset} \{ |B_{\underline{i}}|\}$.\\
 
 \textbf{Claim.} \textit{Let $\underline{i}=(i_1,\dots,i_a)$ be an $a$-tuple and let $x=\sum_{l=1}^a i_l$, then $|B_{\underline{i}}|=f_{a,t}(x)$.\\}
 
 We prove the claim by induction on $a$. If $a=0$ there is nothing to prove. Assuming that the claim holds true for any $a'\leq a-1$, we prove for $a$. Since $B_{\underline{i}}$ contains all $a$-tuples $\underline{j}=(j_1,\dots,j_a)$ such that $\sum_{l=1}^a j_l=x$ then the cardinality of $B_{\underline{i}}$ coincides with the number $a$-partitions of $x$ minus the partitions involving numbers greater or equal than $t+1$. The cardinality of the former set is exactly $\pkn{a}{x}$ while the cardinality of the latter can be calculated as follows. The integers greater or equal than $t+1$ appearing in the partitions are the integers between $t+1$ and $x-(a-1)$. So the number of $a$-partitions of $x$ involving numbers greater or equal than $t+1$ is equal to the sum (over $k$) of $(a-1)$-partitions of $x-k$ involving numbers less or equal than $k$ which, by the inductive hypothesis, is equal to $f_{a-1,k}(x-k)$.\\
%  An integer $y\in[t+1,x-(a-1)]$ appears in the partitions as many times as the possible $a-1$-partitions of $x-y$. Therefore, we have to remove $\sum_{y=t+1}^{x-y}\pkn{a-1}{y}=\sum_{l=a-1}^{x-t-1}\pkn{a-1}{l}$ of them.\\
 
By the previous claims, we have that
 \begin{equation*}
 \max_{B\subset\aupleset}\{|B^\sharp|\} = \max_{\underline{i}\in\aupleset} \{ |B_{\underline{i}}|\}=\max_{\underline{i}\in\aupleset} \left\{f_{a,t}\left(\sum_{l=1}^a i_l\right)\right\}=\max_{a\leq x\leq at} \{f_{a,t}(x)\}.
 \end{equation*}
\end{proof}

\begin{definition}
Let $n$ be a natural number, then define
\begin{equation*}
n_q =
\begin{cases}
1+q+q^2+\dots+q^{n-1} \text{ if } n\neq0\\
0 \text{ otherwise.}
\end{cases}
\end{equation*}
Then
\begin{equation*}
n_q!=n_q\cdot(n-1)_q\cdot\dots\cdot2_q\cdot1_q,
\end{equation*}
and
\begin{equation*}
{n_q \choose k_q}=\frac{n_q!}{k_q!(n-k)_q!}.
\end{equation*}
\end{definition}

\begin{remark}\label{rem-congiungimento-notazionale}
 Let $\underi$ be an element of $\aupleset$ and let $x$ be an integer in $[a,at]$, then we define
\begin{equation*}
\sigmatiny\underi=\sum_{k=1}^a i_k 
\end{equation*}
and
\begin{equation*}
\sigmaAtx{a}{t}{x}=\{\underi\in\aupleset \,|\, \sigmatiny\underi=x\}.
\end{equation*}
With the notation of the proof of Theorem \ref{teo-max-num-massimali} it is easy to see that the set $B_{\underi}$ is exactly the set $\sigmaAtx{a}{t}{x}$ if $\sigmatiny\underi=x$. Therefore, the second claim of the proof of Theorem \ref{teo-max-num-massimali} shows that $\fatx{a}{t}{x}=|B_{\underi}|=|\sigmaAtx{a}{t}{x}|$, where $|\cdot|$ denotes the cardinality of a set.
\end{remark}

\begin{theorem}[\text{\cite[Theorem 3.1]{Andrews}}]\label{teo-quantic-binomial-expression}
Let $n,k$ be integers, then
\begin{equation*}
{(n+k)_q \choose n_q}=\sum_{\lambda \in n\times k} q^{\sigmatiny\lambda}
\end{equation*}
where $n\times k$ are the $\leq n$-tuple with coefficients in $k$, i.e.  $n\times k=\{(i_1,\dots,i_s) \,|\, s\leq n \text{ and } i_l\leq k \text{ for any }l=1,\dots,s\}$.
\end{theorem}

\begin{corollary}\label{cor-coefficients-of-the-quantic-binomial}
 The number $|\sigmaAtx{a}{t}{x}|$ of $s$-tuple with $s\leq a$ with coefficients in $\ttIbar$ whose sum is equal to $x$ is equal to the coefficient of $q^x$ of ${(a+t)_q \choose a_q}$.
\end{corollary}

\begin{theorem}\label{teo-f-and-the-quantic-binomial}
\begin{equation*}
{(a+t-1)_q \choose a_q}=\sum_{x=a}^{at} \fatx{a}{t}{x} q^{x-a}
\end{equation*}
\end{theorem}
\begin{proof}
 Using the results of Theorem \ref{teo-quantic-binomial-expression} and Corollary \ref{cor-coefficients-of-the-quantic-binomial}, one gets the following chain of equalities:
 \begin{equation*}
  {(a+t-1)_q \choose a_q} = \sum_{\lambda \in a\times(t-1)} q^{\sigmatiny\lambda} = \sum_{s=0}^{a(t-1)} \beta_s q^s,
 \end{equation*}
 where $\beta_s=\sum_{l=0}^a|\sigmaAtx{l}{t-1}{s}|$ and $|\cdot|$ denotes the cardinality of a set.
 
 Note that, the number of $a$-tuple with coefficients in $\{1,\dots,t\}$ and whose sum is $x$, is equal to the number of $s$-tuple for $0\leq s\leq a$ with coefficients in $\{1,\dots,t-1\}$ and whose sum is $x-a$, i.e.
 \begin{equation*}
 |\sigmaAtx{a}{t}{x}|=\sum_{l=0}^a|\sigmaAtx{l}{t-1}{x-a}|.
 \end{equation*}
 The left-side of the previous equation is equal to $\fatx{a}{t}{x}$ (Remark \ref{rem-congiungimento-notazionale}), while the right-side is equal to $\beta_{x-a}$ (Corollary \ref{cor-coefficients-of-the-quantic-binomial}). Therefore
  \begin{equation*}
 \sum_{s=0}^{a(t-1)} \beta_s q^s = \sum_{x=a}^{at} \beta_{x-a} q^{x-a} = \sum_{x=a}^{at} \fatx{a}{t}{x} q^{x-a}.
 \end{equation*}
\end{proof}

\begin{corollary}
Let $\ttI$, $\aupleset$ and $f_{a,t}(x)$ be as in Theorem \ref{teo-max-num-massimali}. Then
 \begin{equation*}
  {a+t-1 \choose a}=\sum_{x=a}^{at} f_{a,t}(x).
 \end{equation*}
\end{corollary}
\begin{proof}
Follows directly from Theorem \ref{teo-f-and-the-quantic-binomial} when $q=1$.
\end{proof}

\begin{corollary}
 The maximum of $f_{a,t}(x)$ is attained for $x=\lfloor\frac{a(t+1)}{2}\rfloor$.
\end{corollary}
\begin{proof}
 Follows directly from the properties of the quantic binomial.
\end{proof}

\begin{proposition}\label{prop-formula-fatx}
Let $a,t,x$ as usual. Then
\begin{equation*}
\fatx{a}{t}{x} = \fatx{a}{t-1}{x}+\fatx{a-1}{t}{x-t},
\end{equation*}
or, equivalently,
\begin{equation*}
\fatx{a}{t}{x} = \fatx{a}{t+1}{x}-\fatx{a-1}{t+1}{x-t-1}
\end{equation*}
\end{proposition}
\begin{proof}
From \eqref{eq-function} one gets that
\begin{equation*}
\pkn{a}{x}=\fatx{a}{t}{x}+\sum_{k=t+1}^{x-a+1}\fatx{a-1}{k}{x-k}.
\end{equation*}
Since the left-side of the previous equation does not depend on $t$ we can change $t$ with $t-1$ and get the equality
\begin{equation*}
\pkn{a}{x}=\fatx{a}{t-1}{x}+\sum_{k=t}^{x-a+1}\fatx{a-1}{k}{x-k}.
\end{equation*}
Therefore,
\begin{equation*}
\fatx{a}{t-1}{x}+\sum_{k=t}^{x-a+1}\fatx{a-1}{k}{x-k} = \fatx{a}{t}{x}+\sum_{k=t+1}^{x-a+1}\fatx{a-1}{k}{x-k},
\end{equation*}
and the thesis follows.
\end{proof}

\section{Rank $3$ tensor sheaves on the projective line}\label{sec-rank-3-tensors-on-p1}

In this section we want to describe all degree zero rank $3$ slope semistable tensors $(E,\varphi)$ of type $(3,b,\odi{\pp^1})$ on $\pp^1$, where $\odi{\pp^1}$ is the  trivial bundle.

Since we are on $\pp^1$, the bundle $E$ decompose as $E=L_1\oplus L_2 \oplus L_3$. We will denote by $d_i$ the degree of $L_i$, by $\kkvoid_i$ the number $\kk{L_i}{E}$ and by $d_{ij},c_{ij}$ and $\kkvoid_{ij}$ the corresponding invariants for the bundles $L_i\oplus L_j$. Note that, in this setting, the bundles $E^{\otimes3}$ decomposes as:
\begin{align*}
E^3= L_1^{\otimes3}\oplus L_1^{\otimes3}\oplus L_1^{\otimes3}\oplus 3(L_1^{\otimes2}\otimes L_2)\oplus 3(L_1^{\otimes2}\otimes L_3) \oplus 3(L_2^{\otimes2}\otimes L_1)\oplus\\
\oplus3(L_2^{\otimes 2}\otimes L_3)\oplus 3(L_3^{\otimes2}\otimes L_1)\oplus 3(L_3^{\otimes2}\otimes L_2) \oplus 6(L_1\otimes L_2\otimes L_3).
\end{align*}

As we have already remarked, if the matrix associated to a filtration has only one pivot, then the filtration splits (see Corollary \ref{cor-un-pivot-slope-ss}) and so $L_\ttI^\mu(\efiltration)=\sum_{i\in\ttI} L_{\{i\}}^\mu(0\subset E_i\subset E, \alpha_i)$ and, if the filtration destabilizes, then at least one element $E_i$ of the filtration $\kkvoid$-destabilizes. Thanks to Theorem \ref{teo-max-num-massimali}, the maximum number of pivots in the case $a=r=3$ is $\max_{3\leq x\leq 9}f_{3,3}(x)$. The maximum of $f_{3,3}(x)$ is attained for $x=6$ and $f_{3,3}(6)=2$ therefore, in this case, the matrix associated to any filtration has one or at most two pivots, no matter what $(E,\varphi)$ is. A simple calculation shows that all possible matrices having two pivots are the following:
\begin{enumerate}
\item $\pivot_1=(1,1,3)$ and $\pivot_2=(1,2,2)$
\item $\pivot_1=(1,1,3)$ and $\pivot_2=(2,2,2)$
\item $\pivot_1=(1,2,3)$ and $\pivot_2=(2,2,2)$ 
\item $\pivot_1=(1,3,3)$ and $\pivot_2=(2,2,2)$
\item $\pivot_1=(1,3,3)$ and $\pivot_2=(2,2,3)$
\end{enumerate}

Clearly the semistability condition depends on the morphism $\varphi$, for example, if $\varphi=0$ then $E$ must be a semistable bundle in the usual way hence we obtain $d_1=d_2=d_3$. So, from now on, we will assume that $\varphi$ is not identically zero. We start from the following easy but fundamental result.

\begin{lemma}
Let us assume that the semistability conditions for the tensor $(E,\varphi)$ hold for any filtartion obtained starting from the line bundles $L_i$, id est, all the filtrations
\begin{equation*}
0\subset L_i\subset L_i\oplus L_j \subset E
\end{equation*}
then the tensor $(E,\varphi)$is semistable.
\end{lemma}
\begin{proof}
Let 
\begin{equation*}
(\efiltration)=(0\subset E_1\subset E_2\subset E,(\alpha_1,\alpha_2))
\end{equation*}
any weighted filtration. We want to prove that there exist indexes $i,j\in\{1,2,3\}$ such that
\begin{equation*}
L^\mu(\efiltration;\varphi) \geq L^\mu(0\subset L_i\subset L_i\oplus L_j\subset E,\alphafil;\varphi).
\end{equation*}
Since $\rk{E_1}=1$, without loss of generality, we can assume that exists an injection $g:E_1\to L_1$. Clearly the cokernel of this morphism is a rank zero sheaf, hence, if we have a non zero morphism $\varphi:L_1^{\otimes s}\otimes E^{\otimes a-s}\to\odi{\pp^1}$, then the restriction $\rest{\varphi}{E_1}:E_1^{\otimes s}\otimes E^{\otimes a-s}\to\odi{\pp^1}$ will be non zero as well, that is $\kk{L_1}{E}=\kk{E_1}{E}$. A similar argument works also for the rank $2$ sheaf $E_2$ after considering the decomposition $E_2=M_1\oplus M_2$, where $M_i$ are rank one sheaves. More precisely, up to changing the order of $L_i$, we can assume that $M_i\subset L_i$. So we have proven that 
\begin{equation*}
\mu(E^\bullet,\alphafil;\varphi)= \mu(0\subset L_i\subset L_i\oplus L_j\subset E,\alphafil;\varphi)
\end{equation*}
Moreover,since we have a non zero morphism  $g:E_1\to L_1$ we have that $\deg(E_1)\leq \deg(L_1)$, and $\deg(E_2)\leq \deg(L_1\oplus L_2)$. So we have 
\begin{equation*}
L(E^\bullet,\alphafil)\geq L(0\subset L_i\subset L_i\oplus L_j\subset E,\alphafil)
\end{equation*}
which concludes the proof.
\end{proof}
From now on, we will denote by $(i,ij)$ the following filtration
\begin{equation*}
0\subset L_i\subset  L_i\oplus L_j \subset E,
\end{equation*}
indexed by $\ttI=\{1,2\}$.  So, for example, the element $(1,2,3)$ of the associated matrix describes the beahvour of $\varphi$ restricted to the bundle $L_i\otimes (L_i\oplus L_j) \otimes E$.
 
If $E$ is the trivial bundle $E=\odi{\pp^1}^3$, then the semistability conditions given for the subbundles $L_i=\odi{\pp^1}$ imply that $\kkvoid_i\geq 1$. Let consider now the filtration $(i,ij)$. An easy computation shows that a such filtration is critical if and only if $k_j=3$, and in this case the pivots are $(1,2,3)$ and $(2,2,2)$ and in both cases the semistability conditions are satisfied.\\

So now we assume that E is not trivial, id est $E\neq \ox^3$. We choose the indexes in a such way that $d_1\leq d_2 \leq d_3$, hence, since the degree of $E$ is $0$, all the possible cases are the following:
\begin{enumerate}
 \item[(i)] $d_1<0<d_2<d_3$;
 \item[(ii)] $d_1<d_2\leq0<d_3$;
 \item[(iii)] $d_1=d_2<d_3$;
 \item[(iv)] $d_1<d_2=d_3$.\\
\end{enumerate}

Since $E$ is semistable, considering the semistability condition given by subsheaves, id est, the slope $\kkvoid$-semistability, we obtain the following conditions for $i=1,2,3$:
\begin{equation*}
d_i-\kkvoid_i\ol{\delta} \leq \frac{d}{3}-\ol{\delta} =-\ol{\delta}
\end{equation*}
So in particular, $\kkvoid_3\geq 2$. However, since $d_3>0$ there are no non-zero morphisms $L_3^{\otimes 3}\to\odi{\pp^1}$, hence $\kkvoid_3=2$, so, by definition of $\kkvoid$, the restriction of $\varphi$ to $L_3^{\otimes2}\otimes E$ gives a non-zero morphism $(L_3^{\otimes2}\otimes E)^{\oplus b}\to\odi{\pp^1}$, that is a non-zero morphism:
\begin{equation*}
 \left((L_3^{\otimes 2}\otimes L_1)\oplus (L_3^{\otimes 2}\otimes L_2)\oplus L_3^{\otimes 3}\right)^{\oplus b}\longrightarrow\odi{\pp^1}
\end{equation*}
Since $\deg(L_3^{\otimes 3})$ and $\deg(L_3^{\otimes 2}\otimes L_2)$ are strictly positive, we get that the only non-zero component of the previous morphism must start from $(L_3^{\otimes 2}\otimes L_1)^{\oplus b}$ and this implies that $2d_3+d_1\leq 0$. Since $d_3\geq d_2$ and by assumption $d_1+d_2+d_3=0$, the only possibility is that $d_2=d_3=-\frac{d_1}{2}$. Moreover, the existence of a non zero morphism $(L_3^{\otimes 2}\otimes L_1)^{\oplus b}\to\odi{\pp^1}$ implies that $\kkvoid_1\geq1$.\\

Case $\kkvoid_1=1$. The non-equivalent filtrations we have to consider are the filtrations $(1,12)$, $(2,12)$ and $(3,23)$. Let us consider first the filtration $(1,12)$, that is, the filtration:
\begin{equation*}
0\subset L_1\subset L_1\oplus L_2 \subset E.
\end{equation*}
Since $\kkvoid_1=1$, then one pivot is $(1,2,2)$, and the only possibility for another pivot is in the position $(1,1,3)$; however this element of the matrix associated to $\varphi$ is zero; otherwise, $\kkvoid_1$ should be $2$, So the filtration is not critical. Now consider the filtrations $(2,12)$ and $(3,23)$, that is, the filtrations
\begin{equation*}
0\subset L_2\subset L_1\oplus L_2 \subset E \text{ and } 0\subset L_2\subset L_2\oplus L_3 \subset E.
\end{equation*}
In the first case one pivot is $(1,1,2)$ while, in the second case, it is $(1,1,3)$, but in either cases the filtrations are not critical.\\

Case $\kkvoid_1\geq2$. In this case, one pivot of any of the previous filtrations is $(1,1,2)$; hence, there are not critical filtrations.

%%%%%%%%%%%%%%%%%%%%%%%%%%%%%%%%%%%%%%%%%%%%%%%%%%%%%%%%%%%%%%%%%%%%%%%%%%%%%%%%%%%%%%%%%%%%%%%%%%%%%%%%%%%%%%%%%%%%%%%%
%%%%%%%%%%%%%%%%%%%%%%%%%%%%%%%%%%%%%%%%%%%%%%%%%%%%%%%%%%%%%%%%%%%%%%%%%%%%%%%%%%%%%%%%%%%%%%%%%%%%%%%%%%%%%%%%%%%%%%%%
%%%%%%%%%%%%%%%%%%%%%%%%%%%%%%%%%%%%%%%%%%%%%%%%%%%%%%%%%%%%%%%%%%%%%%%%%%%%%%%%%%%%%%%%%%%%%%%%%%%%%%%%%%%%%%%%%%%%%%%%
%%%%%%%%%%%%%%%%%%%%%%%%%%%%%%%%%%%%%%%%%%%%%%%%%%%% BIBLIOGRAFIA %%%%%%%%%%%%%%%%%%%%%%%%%%%%%%%%%%%%%%%%%%%%%%%%%%%%%%
%%%%%%%%%%%%%%%%%%%%%%%%%%%%%%%%%%%%%%%%%%%%%%%%%%%%%%%%%%%%%%%%%%%%%%%%%%%%%%%%%%%%%%%%%%%%%%%%%%%%%%%%%%%%%%%%%%%%%%%%
%%%%%%%%%%%%%%%%%%%%%%%%%%%%%%%%%%%%%%%%%%%%%%%%%%%%%%%%%%%%%%%%%%%%%%%%%%%%%%%%%%%%%%%%%%%%%%%%%%%%%%%%%%%%%%%%%%%%%%%%
%%%%%%%%%%%%%%%%%%%%%%%%%%%%%%%%%%%%%%%%%%%%%%%%%%%%%%%%%%%%%%%%%%%%%%%%%%%%%%%%%%%%%%%%%%%%%%%%%%%%%%%%%%%%%%%%%%%%%%%%

\bigskip
\end{document}